\documentclass[12pt,a4paper,reqno]{amsart}

\usepackage[utf8]{inputenc}
\usepackage[main=british,ngerman]{babel}
\usepackage{amsmath,amsthm,amssymb,mathtools}
\mathtoolsset{centercolon}
\usepackage{hyperref,float,caption}
\hypersetup{
	colorlinks,
	linkcolor={red!60!black},
	citecolor={green!60!black},
	urlcolor={blue!60!black},
}
\usepackage{geometry}
\usepackage[nameinlink]{cleveref}
\usepackage{enumitem}
\usepackage[dvipsnames]{xcolor}
\usepackage[abbrev, msc-links, alphabetic, nobysame]{amsrefs}

\newtheorem{theorem}{Theorem}[section]
\newtheorem{lemma}[theorem]{Lemma}
\newtheorem{proposition}[theorem]{Proposition}
\theoremstyle{definition}
\newtheorem{question}[theorem]{Question}
\theoremstyle{remark}
\newtheorem{claim}{Claim}
\def\lqedsymbol{\ifmmode$\lrcorner$\else{\unskip\nobreak\hfil
		\penalty50\hskip1em\null\nobreak\hfil$\rule{1.2ex}{1.2ex}$
		\parfillskip=0pt\finalhyphendemerits=0\endgraf}\fi} 

\newenvironment{claimproof}[1][\proofname]
{%
	\proof[#1]%
}
{%
	\endproof%
}

\makeatletter
\@addtoreset{claim}{theorem}
\makeatother

\crefname{section}{section}{sections}
\crefformat{section}{#2Section~#1#3}
\Crefformat{section}{#2Section~#1#3}

\crefname{lemma}{lemma}{lemmata}
\crefformat{lemma}{#2Lemma~#1#3}
\Crefformat{lemma}{#2Lemma~#1#3}

\crefname{theorem}{theorem}{theorems}
\crefformat{theorem}{#2Theorem~#1#3}
\Crefformat{theorem}{#2Theorem~#1#3}

\crefname{claim}{claim}{claims}
\crefformat{claim}{#2Claim~#1#3}
\Crefformat{claim}{#2Claim~#1#3}

\crefname{proposition}{proposition}{propositions}
\crefformat{proposition}{#2Proposition~#1#3}
\Crefformat{proposition}{#2Proposition~#1#3}

\crefname{question}{question}{questions}
\crefformat{question}{#2Question~#1#3}
\Crefformat{question}{#2Question~#1#3}

\newcommand{\braces}[1]{\lbrace #1 \rbrace}
\newcommand{\verts}[1]{\lvert #1 \rvert}
\newcommand{\mcalI}{\mathcal{I}}
\newcommand{\mcalP}{\mathcal{P}}
\newcommand{\mcalQ}{\mathcal{Q}}
\newcommand{\mbbN}{\mathbb{N}}
\newcommand{\mbbQ}{\mathbb{Q}}
\newcommand{\GM}{G_M}
\newcommand{\NM}{N_M}
\newcommand{\GB}{G_B}
\newcommand{\Kalephzero}{K^{\aleph_0}}
\newcommand{\NameToWhereWePlay}{x}

\newcommand{\KQ}{K^\mbbQ}
\newcommand{\Pleq}{\leq_\mcalP}

\title[The $\mathbb{Q}$ game]{The Rational Number Game}
\author[N.\ Bowler]{Nathan Bowler}
\author[F.\ Gut]{Florian Gut}

\address{Universit\"at Hamburg, Department of Mathematics, Bundesstrasse 55 (Geomatikum), 20146 Hamburg, Germany}
\email{\{nathan.bowler, florian.gut\}@uni-hamburg.de}

\begin{document}
	
	\begin{abstract}
		We investigate a game played between two players, Maker and Breaker, on a countably infinite complete graph where the vertices are the rational numbers.
		The players alternately claim unclaimed edges.
		It is Maker's goal to have after countably many turns a complete infinite graph contained in her coloured edges where the vertex set of the subgraph is order-isomorphic to the rationals.
		It is Breaker's goal to prevent Maker from achieving this.
		We prove that there is a winning strategy for Maker in this game.
		We also prove that there is a winning strategy for Breaker in the game where Maker must additionally make the vertex set of her complete graph dense in the rational numbers.
	\end{abstract}
	
	\maketitle
	\section{Introduction}
	Games have always fascinated mathematicians and thus the study of games is a well-established part of the field of combinatorics.
	A particular type of game that falls squarely into graph theory is called the $H$\emph{-building game}.
	This is a game played by two players on a complete graph $G$ as the board.
	The players choose a graph $H$ at the beginning of the game and during the course of the game they alternately claim edges of $G$.
	The first player that has a copy of $H$ contained in the subgraphs of $G$ induced by their respective claimed edges wins the game.
	A comprehensive introduction to games on graphs and other combinatorial games can be found in Beck's book \cite{B08}.
	
	If one considers $H$-building games where $G$ is finite (and therefore also $H$), then one quickly finds that there is always a winning strategy for the first player as long as $G$ is large enough in relation to $H$:
	by Ramsey's theorem~\cite{R30}, there is a number $n=R(\verts{V(H)})$ such that for any $m \geq n$ every colouring of the edges of $K^m$ with two colours contains a copy of $H$ in one of the colour classes.
	Thus, one of the players must have a winning strategy.
	Another classic result called \emph{strategy stealing} (see \cite{B08}) implies that this can only be the first player.
	Due to this the existence of a winning strategy for the first player is known for many games, albeit with no insight into a precise sequence of moves that secure a win for the first player.
	
	A possible approach to circumvent the shortcuts provided by Ramsey theory and strategy stealing is to consider infinite $H$-building games instead:
	on the one hand one may assume that the board $G$ is infinite while $H$ is finite, on the other hand one could also allow $H$ to be infinite.
	In either case the analysis becomes more difficult, as explicit winning strategies must be found.
	This gets tough already for small graphs $H$, as illustrated in \cite{BG23K4}, in which the authors provide a winning strategy for the first player in the $K^4$-building game.
	Thus far, $K^4$ is the largest complete graph for which a winning strategy on the infinite complete board is known.
	
	As the search for winning strategies appears to be demanding in infinite $H$-building games, one may instead consider the \emph{Maker-Breaker} version of the game.
	In that game the premises are the same but only one of the players, Maker, has the aim to complete a copy of $H$ while the other player, Breaker, simply has the objective to stop Maker from completing her copy of $H$.
	
	While for $H$-building games on an infinite board Ramsey's theorem no longer ensures the existence of a winning strategy, there still is a close connection to Ramsey theory (see \cite{GRS90} by Graham, Rothschild and Spencer for an introduction):
	Edges of the board are coloured with two colours and the players want to find a designated substructure completely within their colour classes, in contrast to Ramsey theory where all edges of a graph are coloured and one wants to find the relevant substructure in one of the colours.
	Therefore these types of games are also called \emph{Ramsey games}.
	
	Given the close relationship of Ramsey theory and Ramsey games, it is natural to ask how close the relationship is exactly.
	For this purpose consider a generalisation of the $H$-building game which we call the \emph{structural $H$-building game}:
	One may introduce further structural information into the board $G$ and then require that $H$ satisfies an additional property with regards to the structural information of $G$.
	For example one may consider an order on the vertices of $G$ and require that the subset of vertices that form Maker's copy of $H$ is order-isomorphic to $V(G)$.
	With this we return to the relationship of Ramsey theory and Ramsey games.
	There are two sides to the question how closely the two fields are related.
	\begin{question}\label{ques:Winning_strategy_implied_by_Ramsey}
		Let $G$ be a complete graph with a structural property and $H$ be a graph with a structural property of its vertex set that is compatible with that of $V(G)$.
		Suppose that for any $2$-colouring of the edges of $G$ there is a monochromatic copy of $H$ contained in $G$ as a subgraph.
		Is there a winning strategy for Maker in the structural $H$-building game on $G$?
	\end{question}
	So far all research seems to be supportive of this assertion, the authors are not aware of any example disproving \cref{ques:Winning_strategy_implied_by_Ramsey}.
	\begin{question}\label{ques:Ramsey_implied_by_winnging_strategy}
		Let $G$ be a complete graph with a structural property and $H$ be a graph with a structural property of its vertex set that is compatible with that of $V(G)$.
		Suppose that there is a winning strategy for Maker in the structural $H$-building game on $G$.
		Does this imply that for any $2$-colouring of the edges of $G$ there is a monochromatic copy of $H$ contained in $G$ as a subgraph?
	\end{question}
	A good example for this is the $K^4$-building game.
	For $n \geq 17$ the $K^4$-building game on $K^n$ is a first player win \cite{BG23K4} and the Ramsey number of $4$ is $18$~(see Greenwood and Gleason in \cite{GG55}).
	In particular for every board size of $n \geq R(4)$ there is a winning strategy for the first player.
	Notably there is a small discrepancy between the sufficient board size of $17$ and the Ramsey number of $18$.
	(Note that in \cite{BG23K4} there is no structural property considered but one could e.g. assume that the vertices of the board $\Kalephzero$ and the vertices of $K^4$ are totally ordered, which gives the same result.)
	
	Similarly in the Maker-Breaker version of the $\Kalephzero$-building game played on $\Kalephzero$ there is a winning strategy for Maker \cite{BEG23}*{Theorem 3.1} and for any $2$-colouring of the countably infinite complete graph there is a monochromatic $\Kalephzero$, see e.g. Dushnik and Miller \cite{DM41}*{Theorem 5.22} for a proof of Ramsey's theorem for infinite graphs.
	
	In this paper we answer \cref{ques:Ramsey_implied_by_winnging_strategy} in the negative.
	We present a game and a winning strategy for Maker in that game where the corresponding statement in Ramsey theory is not true.
	
	Before stating the main result of this paper let us shed some light on another disparity of finite and infinite graphs which has to do with total orders.
	While any two total orders on a given finite set are isomorphic, this is not true for infinite total orders.
	For example there are two non-isomorphic total orders on the rational numbers:
	The usual order and one induced by an enumeration of the rationals.
	We carry this over to the setting of structured $H$-building games, and make use of this discrepancy:
	Let $\KQ$ be the complete graph with the rational numbers $\mbbQ$ as the vertices. 
	In the \emph{$\KQ$-building} game we call $\KQ$ the \emph{board}, of which the two players Maker and Breaker alternately claim edges.
	The aim of Maker is to have contained in the subgraph induced by her claimed edges a copy of the board.
	That is, a complete graph such that its vertex set (which is a subset of $\mbbQ$) with the order induced by $\mbbQ$ is order isomorphic to $\mbbQ$.
	It is Breaker's goal to stop Maker from achieving this.
	We will prove that there is a winning strategy for Maker.
	\begin{theorem}\label{theo:rational_number_game_intro_statement}
		There is a winning strategy for Maker in the $\KQ$-building game.
	\end{theorem}
	The corresponding statement in Ramsey theory is false:
	There is a colouring of the graph $\KQ$ with two colours such that there is no monochromatic copy of $\KQ$ contained in either of the colour classes, which is shown in \cref{prop:Ramsey_wrong_on_rationals_squared}.
	
	In \cref{sec:dense_rational_number_game} we show that the result of \cref{theo:rational_number_game_intro_statement} cannot be made stronger in the following sense:
	consider the \emph{dense $\KQ$-building game}:
	The Maker-Breaker game played on $\KQ$ where it is Maker's goal to finish a copy of the board as in the $\KQ$-building game with the additional property that the vertex set of Maker's copy is dense in the vertex set of the board.
	We prove that there is a winning strategy for Breaker in the dense $\KQ$-building game.
	
	\section{Preliminaries}\label{sec:Preliminaries}
	In this paper any graph theoretic notions not introduced are drawn from Diestel's book \cite{D17}.
	
	The games that we analyse are played between two players, Maker and Breaker.
	The players alternately take \emph{turns} that consist of claiming an edge from the \emph{board} $\KQ$ where $\KQ$ is the complete graph with the rational numbers $\mbbQ$ as the vertex set.
	They claim an edge by choosing an uncoloured edge of the board and colouring it in their respective colour.
	We will refer to Maker by \emph{she} or \emph{her} and similarly we refer to Breaker by \emph{he} or \emph{him}.
	The goal for Maker will be to colour the edges in such a way that her subgraph contains a $\KQ$, possibly with an additional property that we introduce in \cref{sec:dense_rational_number_game}.
	The goal of Breaker always is to simply stop Maker from achieving her defined goal.
	For any point in the game we define $E(\GM)$ to be the edges that have been coloured by Maker up to that turn and $V(\GM)$ to be the subset of vertices of the board that are incident with at least one edge of $E(\GM)$.
	With this we define Maker's subgraph $\GM := (V(\GM),E(\GM))$ of the board.
	We similarly define Breaker's subgraph $\GB$.
	For a vertex $v \in V(\GM)$ we define $\NM(v)$ to be the set of \emph{neighbours of $v$ in} $\GM$.
	When we say that a player \emph{connects} a vertex $v$ to a vertex $w$ in her or his turn then we mean that Maker or Breaker claims that edge in her or his turn respectively.
	
	By a \emph{fresh} vertex $v$ we mean a vertex that is not incident with a claimed edge of either player up to that point of the game, i.e. $v \notin V(\GM) \cup V(\GB)$.
	In Maker's strategy that we will describe in \cref{sec:rational_number_game}, Maker will add at most one fresh vertex to $\GM$ in each of her turns except for the first turn.
	We use this to obtain an enumeration of $V(\GM)$, that is $v_k$ stands for the $k$\textsuperscript{th} vertex added to $\GM$.
	For the first turn we will indicate a choice which of the vertices we call $v_1$ and which $v_2$.
	
	In \cref{sec:dense_rational_number_game} we present a winning strategy for Breaker in a variation of the game.
	This will be a \emph{pairing strategy}, that is we will find a family of disjoint pairs of edges of the board which Breaker will utilise in such a way that whenever Maker claims an edge of one of the defined pairs, then Breaker claims the other in his following turn.
	Such a strategy is then a winning strategy for Breaker if any subgraph of the board with which Maker could win the game includes at least one of the defined pairs.
	
	For a natural number $n\in \mbbN$ we define $[n] := \braces{1,\dots, n}$ and for a finite sequence $(a_1,a_2,\dots, a_n)$ of rational numbers and an element $a \in \mbbQ$ we define $(a_1,a_2,\dots, a_n)^\frown a := (a_1,a_2,\dots, a_n,a)$.
	
	For the rest of this \namecref{sec:Preliminaries} we give proofs for three folklore results.
	We prove them here to keep this paper self-contained.
	\begin{proposition}\label{prop:ordered_collection_of_intervals}
		There is a partition $\mcalP$ of $\mbbQ$ into pairwise disjoint open intervals and an order $\Pleq$ of $\mcalP$ defined by $P \Pleq Q$ if and only if either $P = Q$ or $p \leq q$ for every $p \in P$ and $q \in Q$ such that $(\mcalP,\Pleq)$ is order isomorphic to $(\mbbQ,\leq)$.
	\end{proposition}
	Note that since any open interval of $\mbbQ$ is again isomorphic to $\mbbQ$ (since any two countable dense total orders without smallest and largest element are isomorphic, see \cite{BMMN06}*{Theorem 9.3}), there also exists such a collection for any open interval of $\mbbQ$.
	\begin{proof}
		Let $\preceq$ be the lexicographic order on $\mbbQ^2$ and $f \colon (\mbbQ^2, \preceq) \to (\mbbQ,\leq)$ be an order isomorphism.
		Then $\mcalP := \braces{f[\braces{q}\times \mbbQ] \colon q \in \mbbQ}$ is as desired as $\Pleq$ inherits its properties from $\leq$.
	\end{proof}
	\begin{proposition}\label{prop:Ramsey_true_on_rationals}
		For every $2$-colouring of the rational numbers there is a monochromatic subset that is isomorphic to the rational numbers.
	\end{proposition}
	\begin{proof}
		Choose a collection $(\mcalP,\Pleq)$ as in \cref{prop:ordered_collection_of_intervals} and let a $2$-colouring of $\mbbQ$ be given.
		Call the colours red and blue.
		First suppose that there is a set $P\in \mcalP$ that contains no red element.
		Then $P$ is an open blue interval, which is isomorphic to $(\mbbQ,\leq)$ by Cantor's isomorphism theorem \cite{BMMN06}*{Theorem 9.3}.
		Now suppose that every $P \in \mcalP$ contains a red element.
		Then choose such an element from each $P$.
		The union of these elements together with the order $p \leq q \Leftrightarrow P \leq Q$ where $p \in P \in \mcalP$ and $q \in  Q \in \mcalP$, is then isomorphic to $(\mbbQ,\leq)$.
	\end{proof}
	\begin{proposition}\label{prop:Ramsey_wrong_on_rationals_squared}
		There is a $2$-edge-colouring of $\KQ$ such that there is no monochromatic subgraph $\KQ$.
	\end{proposition}
	\begin{proof}
		First note that $\mbbN$ and $\mbbQ$ are not order isomorphic, since $\mbbN$ has a smallest element, but $\mbbQ$ does not.
		We will use this for a contradiction.
		Let $(q_i)_{i\in \mbbN}$ be an enumeration of $\mbbQ$.
		We colour $q_iq_j \in E(\KQ)$ blue if either $q_i \leq q_j$ and $i \leq j$ or $q_j \leq q_i$ and $j \leq i$.
		Otherwise we colour it red.
		Now suppose that there is a subset $Q$ of $\mbbQ$ that is isomorphic to $(\mbbQ,\leq)$ such that the complete graph induced by $Q$ is monochromatic.
		This implies that $Q$ is also order-isomorphic to $\mbbN$ or its reversal, a contradiction.
	\end{proof}
	
	\section{The rational number game}\label{sec:rational_number_game}
	In this \namecref{sec:rational_number_game} we present a winning strategy for Maker in a variant of the $\Kalephzero$-game.
	We will do this by first describing a strategy according to which Maker should play and then showing in the proof of \cref{theo:rational_number_game} that there actually is a subgraph with the desired property if Maker adheres to the strategy.
	During the game Maker wants to switch between adding fresh vertices to $\GM$ and further connecting vertices of $\GM$ to others in a structured fashion.
	We will make this more precise later.
	
	Let us define the \emph{$\mbbQ$-game strategy}:
	At the beginning of the game Maker chooses 
	\begin{itemize}
		\item a partition $\mcalP$ of the interval $(0,1) \subseteq \mbbQ$ into open intervals such that there is an order $\leq_{\mcalP}$ of $\mcalP$ induced by the regular order on $\mbbQ$ such that $(\mcalP, \leq_{\mcalP})$ is order isomorphic to $(\mbbQ,<)$ where ``$<$'' is the usual ordering (this is possible by \cref{prop:ordered_collection_of_intervals}),
		\item an enumeration $(P_i)_{i \in \mbbN}$ of $\mcalP$, and
		\item a sequence $(n_i)_{i \in \mbbN}$ of natural numbers such that every finite sequence of natural numbers appears infinitely often as a consecutive subsequence.
	\end{itemize}
	
	Informally speaking, by containing vertices of sufficiently many different intervals in $\mcalP$, Maker can ensure that a subset of $V(\GM)$ actually is isomorphic to $\mbbQ$, as $(\mcalP,\leq_\mcalP)$ is isomorphic to $(\mbbQ,\leq)$.
	The enumeration $(P_i)_{i \in \mbbN}$ together with the sequence $(n_i)_{i\in\mbbN}$ will be used to ensure that different intervals $P_i$, $P_j$ are well connected to ensure that there actually is a complete graph using vertices of many different $P_i$.
	
	In her first turn, Maker claims the edge $\braces{0,1}$ for herself and sets $v_1 = 0$ and $v_2 = 1$.
	In a later turn, suppose $V(\GM) = \braces{v_1, \dots, v_k}$ where $v_i$ is the $i$\textsuperscript{th} vertex that Maker added to $\GM$.
	For every vertex $v \in V(\GM)$ we define a finite sequence $S_v = (v_{i_1}, \dots , v_{i_\ell}) $ of the vertices of $\NM(v)$ that were added to $\GM$ before $v$ that represents the order in which Maker claimed the edges $vv_{i_j}$.
	The $\ell$ which appears here was the degree of $v$ in $\GM$ when Maker first chose a new fresh vertex after $v$.
	Maker uses the sequence $S_{v_k}$ to determine from which interval $P \in \mcalP$ to choose the vertex that she plays to next:
	suppose that $v_k \in P \in \mcalP$ and $S_{v_k} = (v_{i_1}, \dots, v_{i_\ell})$.
	Define $L := \braces{v \in P \cap V(\GM) \colon S_v \vert_\ell = S_{v_k}}$, i.e. $L$ is the set of vertices of $\GM$ that come from the same partition class of $\mcalP$ as $v_k$ and in their first $\ell$ moves of being connected to $\GM$, Maker connected them to $\GM$ in the same manner as $v_k$.
	Consider the $(\verts{L}+1)$\textsuperscript{st} time that $(i_1,i_2,\dots , i_\ell)$ appears in $(n_i)_{i\in\mbbN}$ as a subsequence and let $n$ be the number appearing next.
	Then Maker wants to play to a vertex of $P_n$ next.
	
	For that purpose she considers the vertices $v_i$ such that
	\begin{enumerate}[label={\textbf{(\arabic*)}},ref={(\arabic*)}]
		\item\label{item:strat:lies_in_relevant_partition_class} $v_i \in P_n \cap V(\GM)$,
		\item\label{item:strat:has_played_likewise} $S_{v_i} \vert_\ell = S_{v_k}$, and
		\item\label{item:strat:size_of_set} there are at most $\ell\cdot (\ell+1) $ vertices $v_j\in V(\GM)$  with $j < i$ satisfying \labelcref{item:strat:lies_in_relevant_partition_class,item:strat:has_played_likewise}.
	\end{enumerate}
	Let us call this set $F$.
	If $F$ contains fewer than $\ell (\ell+1)+1 $ vertices, Maker chooses a fresh vertex $v_{k+1} \in P_{n_{k+1}}$, claims $v_1v_{k+1}$ and begins the aforementioned process for $v_{k+1}$.
	Otherwise, $F$ has size precisely $\ell (\ell+1)+1 $ and Maker continues by defining a total order on $F$ via
	\begin{equation}\label{equ:order_to_play_balanced}
		v_i < v_j :\Leftrightarrow \begin{cases} \verts{\NM(v_i) \cap L } < \verts{\NM(v_j) \cap L } &, \text{or} \\ \verts{\NM(v_i) \cap L } = \verts{\NM(v_j) \cap L} \text{ and } i < j & \end{cases}.
	\end{equation}
	Maker claims $v_iv_k$ such that $v_i$ is minimal with respect to that order and $v_iv_k$ is unclaimed.
	Note that this choice is unique.
	In fact the second clause in \labelcref{equ:order_to_play_balanced} is only there to make this true.
	
	Informally, Maker will use this order to play to the vertices of $F$ in a `balanced fashion', that is, Maker connects $v_k$ to an available vertex $v_i$ with the smallest possible number of neighbours that behave like $v_k$ in the sense that $S_{v_i}\vert_\ell = S_{v_k}\vert_\ell$.
	This is useful:
	consider an infinite set $W$ of vertices of which each vertex was treated by Maker just like $v_k$ until it had degree $\ell$ in $\GM$. Then, regardless of how Breaker plays, for all but at most $\ell$ many vertices $v$ of $F$ there are infinitely many vertices in $W$ that get connected to $v$.
	
	\begin{theorem}\label{theo:rational_number_game}
		The $\mbbQ$-game strategy is a winning strategy for Maker in the complete rational number game.
	\end{theorem}
	Note that \cref{theo:rational_number_game} implies \cref{theo:rational_number_game_intro_statement}.
	\begin{proof}
		Suppose that Maker and Breaker, playing the complete rational number game against each other, have just completed a game, in which Maker played according to the $\mbbQ$-game strategy.
		We begin by fixing a sequence $(a_i)_{i \in \mbbN}$ of natural numbers.
		This sequence should contain every number infinitely often and each $a_i$ should be at most $i$.
		Moreover, we reuse the partition $\mcalP$ and its enumeration $(P_i)_{i \in \mbbN}$ from the $\mbbQ$-game strategy.
		We recursively build for every $m \in \mbbN$ a complete graph $K^m$ and a set $W_m$ such that
		\begin{enumerate}[label={\textbf{(\alph*)}},ref={(\alph*)}]
			\item\label{item:complete_graph_chain} $K^{m-1}\subseteq K^m \subseteq \GM$ with $V(K^m) = \braces{u_1,\dots,u_m}$, and
			\item\label{item:reservoir} $W_{m} = \braces{w \in V(\GM) \colon S_w \vert_m = (u_1, \dots , u_m)} $.
		\end{enumerate}
		For every vertex $u_i$ in $(u_1, \dots, u_m)$ we denote by $R_{i}$ the partition class of $\mcalP$ that contains $u_i$.
		We define
		\begin{align}\label{equ:definition_of_intervals}
			\begin{split}
				\mcalQ^{i}_m := \braces{ P\in \mcalP \colon & \verts{W_m \cap P} = \aleph_0 \text{ and } R_{i} <_{\mcalP} P  \text{ and} \\
					&\text{there is no } u_j \text{ with } R_{i} <_{\mcalP} R_{j} <_{\mcalP} P}.
			\end{split}
		\end{align}
		With these definitions we also require $W_m$ to satisfy that
		\begin{enumerate}[resume*]
			\item\label{item:Q_everywhere} $\mcalQ^{i}_m$ contains a subset that, together with the order induced by $<_{\mcalP}$, is order isomorphic to $(\mbbQ,<)$ for every $i \in [m]$, and
			\item\label{item:choose_where_to_play} $u_{m}\in \bigcup \mcalQ^{a_{m-1}}_{m-1}$, if $m > 1$.
		\end{enumerate}
		
		Clearly, \labelcref{item:complete_graph_chain} implies that $\bigcup_{m \in \mbbN}K^m = \Kalephzero$ is a complete graph contained in Maker's subgraph $\GM$.
		By \labelcref{item:choose_where_to_play} we add a vertex $u_h$ between any pair of vertices $u_i,u_j \in \Kalephzero$, thus $V(\bigcup_{m \in \mbbN}K^m)$ contains a subset that is order isomorphic to $(\mbbQ,\leq)$.
		The subgraph induced by this subset is the desired $\KQ$.
		Properties \labelcref{item:reservoir,item:Q_everywhere} are needed to ensure that the recursion can be continued indefinitely.
		
		\textbf{Recursion start:}
		Recall that $v_1=0$ was added to $\GM$ as the first vertex according to the $\mbbQ$-game strategy.
		We set $K^1 := (\braces{v_1},\emptyset)$ and $W_1 := V(\GM)\setminus \braces{v_1}$.
		The requirements of \labelcref{item:complete_graph_chain,item:choose_where_to_play} are empty in the initial step and therefore satisfied.
		Property \labelcref{item:reservoir} is true for $W_1$, as according to the $\mbbQ$-game strategy every vertex of $\GM$ other than $v_1$ is connected to $v_1$ first.
		Lastly, for every $P \in \mcalP$ Maker added infinitely many vertices of $P$ to $\GM$, as the $i$\textsuperscript{th} vertex she adds is a vertex of $P_{n_i}$ and $(n_i)_{i\in\mbbN}$ contains every number infinitely often.
		This, together with the definition of the $\mcalQ^i_m$ implies that $\mcalQ^1_1$ contains every $P \in \mcalP $ with $P_i < P$ and therefore $\mcalQ^1_1$ satisfies \labelcref{item:Q_everywhere}.
		
		\textbf{Recursion step:}
		Let $k \geq 1$ and let $K^k$ and $W_k$ satisfying \labelcref{item:choose_where_to_play,item:complete_graph_chain,item:Q_everywhere,item:reservoir} be given.
		We demonstrate how we can find a vertex $u_{k+1}$ and a set $W_{k+1}$ such that $K^{k+1} := G[V(K^k) \cup \braces{u_{k+1}}]$ and $W_{k+1}$ comply with \labelcref{item:complete_graph_chain,item:choose_where_to_play,item:Q_everywhere,item:reservoir}.

		By $\labelcref{item:Q_everywhere}$, $\mcalQ^{a_k}_k$ contains a subset that is order isomorphic to $(\mbbQ,<)$.
		Let ${\mcalQ}$ be such a subset and $\NameToWhereWePlay$ be the first element of $(a_i)_{i \in \mbbN}$ such that $P_\NameToWhereWePlay $ is an element of ${\mcalQ}$ but contains no vertex of $V(K^k)$.
		We choose $\NameToWhereWePlay$ in such a way because this ensures that there is a partition of $\mcalQ^{a_k}_k \setminus \braces{P_\NameToWhereWePlay}$ into two subsets $\mcalQ^{-}$ and $\mcalQ^{+}$, both containing a $\mbbQ$-isomorphic subset respectively, such that all $Q^{-} \in \mcalQ^{-}$ and $Q^{+} \in \mcalQ^{+}$ fulfil $Q^{-} <_{\mcalP} P_\NameToWhereWePlay <_{\mcalP} Q^{+}$.
		We will choose a vertex of $P_\NameToWhereWePlay$ as $u_{k+1}$.
		This choice ensures \labelcref{item:choose_where_to_play}.
		
		Next we want to suitably restrict $W_k$ to a set of vertices that were connected to a vertex of $P_\NameToWhereWePlay$, but we also need to make sure that we can preserve \labelcref{item:Q_everywhere} for the subsequent steps.
		For this purpose we let 
		\begin{align*}
			W' := \braces{w \in W_k \colon \text{there is } v \in P_\NameToWhereWePlay \text{ such that } S_w\vert_{k+1} = (u_1, \dots , u_k, v)}
		\end{align*}
		and use this to define $\hat{\mcalQ}^i$ similarly to the definition in \labelcref{equ:definition_of_intervals}:
		for $i \in [k]$ let $R_i$ be the partition class of $\mcalP$ that contains $u_i$ and set $R_{k+1} := P_\NameToWhereWePlay$.
		With this, for $i \in [k+1]$, we define
		\begin{align*} 
			\hat{\mcalQ}^i := \braces{ P\in \mcalP \colon & \verts{W' \cap P} = \aleph_0 \text{ and } R_{i} <_{\mcalP} P  \text{ and} \\
				&\text{there is no } j\in [k+1] \text{ with } R_{i} <_{\mcalP} R_{j} <_{\mcalP} P}\, .
		\end{align*}
		\begin{claim}\label{cla:Q_i_contains_a_copy_of_Q}
			$\hat{\mcalQ}^i$ contains a subset that is order isomorphic to $(\mbbQ,<)$ for every $i \in [k+1]$.
		\end{claim}
		\begin{claimproof}
			\begin{enumerate}[label=\textbf{Case \arabic*:},ref=Case \arabic*,align = left,noitemsep,leftmargin=6pt]
				\item\label{item:Q_everywhere_claim_regular} \textbf{$i \in [k+1] \setminus \braces{a_k,k+1}$.}
				We prove this case by showing that any element of $\mcalQ^{i}_k$ is also an element of $\hat{\mcalQ}^i$.
				The \namecref{cla:Q_i_contains_a_copy_of_Q} then follows from \labelcref{item:Q_everywhere}.
				
				Consider any $P \in \mcalQ^{i}_k $.
				By definition we have $\verts{P \cap W_k}=\aleph_0$ and by \labelcref{item:reservoir} every $w \in P\cap W_k$ was connected precisely to $V(K^k)$ in its first $k$ moves of being connected to $\GM$.
				According to the $\mbbQ$-game strategy, for infinitely many of them Maker played to a vertex of $P_\NameToWhereWePlay$ next, as $x$ is chosen according to the appearances of the finite sequence $(1,\dots,k,x)$ in the sequence of $(n_i)_{i \in \mbbN}$, which contains every finite sequence infinitely often.
				Thus $\verts{P \cap W'}=\aleph_0$.
				But this is true for any $P \in \mcalP$ with $\verts{P \cap W_k}=\aleph_0$.
				Thus, the \namecref{cla:Q_i_contains_a_copy_of_Q} follows for $i \in [k+1]\setminus \braces{a_k,k+1}$.
				
				\item \textbf{$i \in \braces{a_k,k+1}$.}
				With the same reasoning as in \labelcref{item:Q_everywhere_claim_regular} any $P \in \mcalQ^{-}$ is also in $\hat{\mcalQ}^{a_k}$ and any $P \in \mcalQ^{+}$ is also in $\hat{\mcalQ}^{k+1}$.
				Since both, $\mcalQ^{-}$ and $\mcalQ^{+}$, are isomorphic to $\mbbQ$, this implies the \namecref{cla:Q_i_contains_a_copy_of_Q} for $i \in \braces{a_k,k+1}$. \qedhere 
			\end{enumerate}
		\end{claimproof}

		To continue, consider the set $F$ of the first $k(k+1)+1$ many vertices $u \in R_{k+1}$ with $S_u \vert_k = (u_1 , \dots, u_k)$.
		There is such a set as any vertex $u \in R_{k+1} \cap W_k$ fulfils this property and there are infinitely many such vertices by the choice of $R_{k+1}$.
		Further, fix $i \in [k+1]$ and consider $Q \in \hat{\mcalQ}^i$.
		For any vertex $q \in Q \cap W'$, Maker considered the set $F$ in the $(k+1)$\textsuperscript{st} move of connecting $q$ to $\GM$.
		As Maker played from $Q \cap W'$ to $F$ in a balanced fashion, that is she played to the smallest vertex of $F$ with respect to the order defined in \labelcref{equ:order_to_play_balanced}, there are at most $k$ vertices in $F$ that have only finitely many neighbours in $Q \cap W'$.
		By choosing some superset of $k$ vertices if necessary, we obtain a colouring of $\hat{\mcalQ}^i$ indicating for every $Q\in \hat{\mcalQ}^i$ a subset of $F$ of size $k$ such that all other vertices of $F$ have infinitely many neighbours in $Q\cap W'$.
		By taking the complements of the sets of size $k$, the same colouring indicates for each $Q$ for which subset $F^\prime \subseteq F$ every vertex in $F^\prime$ was picked by Maker for infinitely many vertices $v$ of $Q \cap W'$ to be played to in the $(k+1)$\textsuperscript{st} move of connecting $v$ to $\GM$.
		There are $\binom{k(k+1)+1}{k}$ colours in this colouring of $\hat{\mcalQ}^i$, since every $k$ element subset of $F$ gets assigned a colour and $F$ has $k(k+1)+1$ elements.
		In particular, this is a finite number.
		Thus by \cref{prop:Ramsey_true_on_rationals} and \cref{cla:Q_i_contains_a_copy_of_Q}, there is a colour class that again contains a subset that is order isomorphic to $\mbbQ$.
		We fix a suitable colour class $C^i \subseteq \hat{\mcalQ}^i$ for every $i \in [k+1]$.
		As seen, any of the $k+1$ fixed colour classes excludes $k$ vertices of $F$ as the next candidate and since ${F} $ has size $k(k+1)+1$, there is at least one vertex in $F$ that is met by the fixed colour class $C^i$ for every $i \in \mbbN$.
		We choose the smallest such vertex as $u_{k+1}$.
		Then \labelcref{item:reservoir,item:Q_everywhere} are fulfilled by construction and \labelcref{item:choose_where_to_play} is fulfilled by the choice of $\NameToWhereWePlay$ as mentioned above.
		Lastly, we set $K^{k+1} := G[V(K^k) \cup \braces{u_{k+1}}]$, thus also \labelcref{item:complete_graph_chain} is ensured.
	\end{proof}

	\section{The dense rational number game}\label{sec:dense_rational_number_game}
	
	Since Maker can always win in the rational number game, let us consider a variant of the game in which it is harder for Maker to achieve her goal.
	Thus, consider the \emph{dense rational number game} in which the vertices of the board are again represented by the rational numbers and it is Maker's aim to have at the end of the game a complete graph $\Kalephzero$ contained in her graph $\GM$ such that the set of vertices $V(\Kalephzero)$ is a dense subset of $\mbbQ$.
	
	In this game Maker is doomed to fail.
	We prove this by providing a winning strategy for Breaker in this variant of the game.
	We call it the \emph{dense $\mbbQ$-game strategy}:
	at the beginning of the game Breaker picks
	\begin{itemize}
		\item an infinite sequence $\mcalI := \left( I_j\right)_{j \in \mbbN}$ of pairwise disjoint intervals of $\mbbQ$,
		\item an enumeration $\mcalQ$ of $\mbbQ$, and
		\item an enumeration $(\braces{p_j,q_j})_{j \in \mbbN}$ of the $2$-element subsets of $\mbbQ$.
	\end{itemize}
	For any $s \in I_j \setminus \braces{p_j,q_j}$ where $s$ appears later in the enumeration $\mcalQ$ than both $p_j$ and $q_j$ he pairs the edges $\braces{p_js,q_js}$.
	Then, whenever Maker claims either $p_js$ or $q_js$ in one of her turns, Breaker claims the other in his following turn.
	
	\begin{lemma}
		The dense $\mbbQ$-game strategy is a pairing strategy that is a winning strategy for Breaker in the dense rational number game.
	\end{lemma}
	\begin{proof}
		To verify that the dense $\mbbQ$-game strategy is a pairing strategy we need to verify that the pairs of edges $\braces{p_js,q_js}$ are pairwise disjoint.
		For fixed $j$ and $s \neq t \in I_j \setminus \braces{p_j,q_j}$ it is true that $\braces{p_js,q_js} \cap \braces{p_jt,q_jt} = \emptyset$.
		For $i\neq j$ with $s \in I_i$ and $t \in I_j$, the intersection of $\braces{p_is,q_is}$ and $ \braces{p_jt,q_jt}$ can only be nonempty if both $s \in \braces{p_j,q_j}$ and $t \in \braces{p_i,q_i}$.
		But this cannot happen:
		In the enumeration $\mcalQ$ the vertices $p_i$, $q_i$, $p_j$ and $q_j$ appear in a unique order. 
		In case either $p_i$ or $q_i$ appear the latest, Breaker considers only the pairs $\braces{p_is,q_is}$ and in case either $p_j$ or $q_j$ appear the latest, Breaker only considers $\braces{p_jt,q_jt}$.
		
		To see that the dense $\mbbQ$-game strategy is a winning strategy for Breaker, suppose for a contradiction that Maker finishes a $\KQ$ whose vertices are dense in $\mbbQ$.
		Then there is $n \in \mbbN$ such that $\braces{p_n,q_n} \subseteq V(\KQ)$.
		But according to the dense $\mbbQ$-game strategy only finitely many vertices of $I_n$ can be in $\KQ$:
		since at some point both $p_n$ and $q_n$ have appeared in $\mcalQ$, only finitely many other elements have appeared before and for every vertex $v \in I_n$ that appears later in $\mcalQ$ the pair $\braces{p_n v, q_n v}$ is considered in the dense $\mbbQ$-game strategy.
		Thus, $V(\KQ)$ cannot be dense in $I_n$, which implies that the set of vertices is not dense in $\mbbQ$.
	\end{proof}

\end{document}